\documentclass[journal,twoside,web]{ieeecolor}

\usepackage{amsmath}
\usepackage{amssymb}
\usepackage{mathtools}
\usepackage{enumerate}
\usepackage{graphicx}
\usepackage{textcomp}

\usepackage{algorithm}
\usepackage{algorithmicx}
\usepackage{cite}
\usepackage{mathrsfs}
\usepackage{lcsys}

\newtheorem{Theorem}{Theorem}[section]

\newtheorem{Lemma}[Theorem]{Lemma}
\newtheorem{Corollary}[Theorem]{Corollary}

\newtheorem{Remark}[Theorem]{Remark}

\usepackage{subcaption}
\usepackage{float}
\usepackage{bm}
\usepackage{hyperref}

\begin{document}

\def\BibTeX{{\rm B\kern-.05em{\sc i\kern-.025em b}\kern-.08em
    T\kern-.1667em\lower.7ex\hbox{E}\kern-.125emX}}
\markboth{\journalname, VOL. XX, NO. XX, XXXX 2017}
{Author \MakeLowercase{M. Petreczky \textit{et al.}}}

\title{Stability of input-output maps and their minimal realizations  in  state-linear, state-affine, LPV, and linear switched systems}

\author{Mihály Petreczky, Juan-Pablo Ortega,
Florian Rossmannek and 
Bálint Daróczy
\thanks{M. Petreczky is with Univ. Lille, CNRS,
 Centrale Lille, UMR 9189 CRIStAL, F-59000 Lille, France
  (mihaly.petreczky@centralelille.fr).}
\thanks{J.-P. Ortega and 
F. \textcolor{black}{Rossmannek} are with SPMS, NTU, Singapore
  (e-mail: juan-{pablo.ortega,florian.\textcolor{black}{rossmannek}}@ntu.edu.sg).}  
\thanks{B. Daróczy is with
HUN-REN SZTAKI, Budapest, Hungary (e-mail: daroczy.balint@sztaki.hun-ren.hu).}}


\maketitle 
\thispagestyle{empty}

\begin{abstract}
      Stability is often assumed in learning and identification, yet it is rarely characterized directly from input--output data. We show that an input--output family admits a stable finite-dimensional state-linear realization iff it has finite Hankel-rank and its response decays uniformly with time; for state-linear realizable maps this decay is necessarily exponential. We extend these results to state-affine, LPV, and linear switched systems via suitable input-forgetting notions, and relate forgetting to decay of impulse responses (sub-Markov parameters). In all cases, the decay/forgetting rate determines the decay rate of every minimal realization.
\end{abstract}
\section{Introduction}
In this paper we are interested in characterizing input-output maps
\textcolor{black}{of stable \textcolor{black}{state-linear} systems (SLS)} 
\begin{equation}
    \label{eq1} 
       \begin{split}
         &  x(t+1)=A_{u(t)}x(t), \quad x(0) \in B, \quad  y(t)=C x(t)
      \end{split}
\end{equation}
where $x(t) \in \mathbb{R}^n$ is the state, $y(t) \in \mathbb{R}^{\textcolor{black}{n_y}}$ is the output and $u(t) \in \mathbb{U}$ is the 
\textcolor{black}{input at time $t$}, $A_{u(t)}$ and $C$ are matrices of suitable sizes,
   and $B=\{B_j \in \mathbb{R}^n\}_{j \in J}$
   is a family of initial states indexed by a set $J$.
  \textcolor{black}{SLSs are also known as switched linear systems (with no additive input) when $\mathbb{U}$ is finite or $\{A_u\}_{u \in \mathbb{U}}$ is compact. We prefer the term SLS as it is older \cite{Son:Real} and it is not associated with 
  assumptions on $\mathbb{U}$.}
  %

  \begin{color}{black}
   Stability of \eqref{eq1} is well-understood, without claiming completeness \cite{Chitour1,RJungers,IanMorris2022,Protasov2022,mason2023}. In particular, it is known that the asymptotic stability of \eqref{eq1} is equivalent to uniform \emph{global exponential stability (GUES)} and both are equivalent to the joint spectral radius of $\{A_u\}_{u \in \mathbb{U}}$ being less than one. 
   However, it is less clear how stability properties of the input-output behavior of \eqref{eq1} relate to the stability of \eqref{eq1}, e.g., it is unclear if
   \eqref{eq1} is GUES if its input-output maps decay.
   
  Our \emph{first contribution} is to show that a family of input-output maps has a GUES SLS realization if and only if it has an SLS realization and the values of the input-output maps decay to zero; in that case, decay is necessarily
  exponential, and all minimal realizations are GUES with this decay rate.

  This is relevant for two reasons. 
  First, state-space representations (SSRs) obtained through identification or model reduction capture, at best, only the true input-output behavior, 
  and therefore cannot be identified with the true system itself. 
  Hence, in the absence of prior knowledge, any SSR reproducing the true input-output behavior is an equally valid candidate model.
  Therefore, analysis and control design for one such SSR should apply to 
  other SSRs with the same input-output map, e.g.,
  a controller that stabilizes one SSR (and thus its input-output behavior) should stabilize other SSRs with the same input-output behavior. 
  Our results suggest that this might be the case if the other SSRs are minimal.
  Second, in learning/system identification, stability is often imposed on SSRs to guarantee robust long-term prediction and prevent error accumulation. However, this restriction is only meaningful if the true input-output behavior admits a stable realization. Our results provide conditions for this.

  As a \emph{second contribution} we apply 
  the results for SLSs 
  to \emph{state-affine systems (SAS) \cite{Son:Real}, linear parameter-varying systems with affine dependence on the scheduling variable (LPV-SSA) \cite{Petreczky2016,Toth1}, and linear switched systems (LSS) \cite{Sun:Book}}.
  Realizability by a GUES \emph{SAS }is equivalent to an appropriate input-forgetting property, where a SAS is called GUES if its linearization is GUES. 
  Realizability by a GUES LPV-SSA/LSS, i.e., one for which the SLS obtained by taking zero additive input is GUES, is equivalent to the impulse response (sub-Markov parameters) decaying to zero, and equivalently, to $\ell_\infty$-forgetting (input forgetting with BIBO stability).
%
   In all cases, the decay/forgetting rate is necessarily exponential and it coincides with the decay rate of minimal realizations, which are necessarily GUES. 
  The relevance of these results for SASs/LPV-SSAs/LSSs mirrors that for SLSs. For SASs, they
  are also relevant for reservoir computing \cite{Ortega4,Ortega3,Sepulchre1,GONON202110}, as they relate
  the \textcolor{black}{fading memory and stability/echo-state properties.}
  \end{color}
\textbf{Related work.}
  For linear systems, the relationships among internal stability, BIBO stability, and $\ell_p$-gains are classical results \cite{Cal:Des}. For nonlinear systems, similar connections are subtler and often require strong reachability/observability assumptions \cite{WillemsStabOld,FROMION1996243,van19922,HILL1980327,MoylanCounterExamples,478341}, which may not hold for structured classes like SLS/SAS/LPV-SSA/LSS.
  \begin{color}{black}
  While stability of SLS/SAS/LPV-SSA/LSS state-space representations is well-studied, e.g., \cite{Chitour1,RJungers,IanMorris2022,Protasov2022,mason2023,Toth1,Sun:Book},
  our contribution is to establish the equivalence between decay/input-forgetting of input-output maps
  and stability properties of their minimal realizations, 
  by combining realization theory with known stability results for SSRs.
  \end{color}
  In reservoir computing, 
  fading memory and stability 
  were studied for SASs \cite{Ortega4,Ortega3,Sepulchre1,GONON202110}, 
  \textcolor{black}{but not the correspondence between 
  fading-memory of
  input-output maps and the stability of their
  minimal finite-dimensional realizations. }


\section{Preliminaries}
 A \emph{nonempty word} over $X$ is a finite sequence of elements of $X$, i.e.,  
$w = x_{1}x_{2}\cdots x_{k}$ for some $k \in \mathbb{N}$, $k >0$, $x_{1}, x_{2}, \ldots, x_{k} \in X$; $|w|:=k$ is the 
length of $w$. \textcolor{black}{The set of all nonempty words is $X^{+}$. The empty word is denoted by $\epsilon$, $|\epsilon|=0$, and $X^{*} = \{\epsilon\} \cup X^{+}$. Concatenation is defined by $vw = a_1\cdots a_m b_1\cdots b_n$ for $v = a_1\cdots a_m$, $w = b_1\cdots b_n$, and $v\epsilon = \epsilon v = v$.}
We denote by $X^k$ the set of sequences of length $k$, and by $X^{\infty}$ the set of all infinite sequences of elements of $X$;
for any $w \in X^{\infty}$, $|w|=\infty$.
If $w \in X^{*} \cup X^{\infty}$, and $k \le |w|$ then $w[k]$ denotes the $k$th element of 
$w$, i.e., $w=w[1]\cdots w[|w|]$. \textcolor{black}{If $i > |w|$, then $w[i]\cdots w[|w|]$ is understood to be the empty word $\epsilon$.} 
	\textcolor{black}{If $Z$ is an $n \times m$ matrix, then $Z_{i,j}$ denotes the element on the $i$th row and $j$th column of $Z$; $Z_{i,\cdot}$ and $Z_{\cdot,j}$ denote the $i$th row and $j$th column of $Z$; if $z$ is a vector \textcolor{red}{with $n$ entries}, then $z_i$ denotes the $i$th element of $z$. 
We denote by $\|Z\|_F$ the Frobenius norm of $Z$, by $\|Z\|_2$ the spectral norm of $Z$,  \textcolor{red}{by $\|z\|_2$ the Euclidean norm of $z$, and 
by $\|z\|_{\infty}$ the supremum norm, i.e., $\|z\|_{\infty}=\max_{i,\ldots,n} |z_i|$.}  Let $I_n$ be the $n \times n$ identity matrix.}
\textcolor{black}{For a sequence $w \in (\mathbb{R}^p)^{*} \cup (\mathbb{R}^p)^{\infty}$ 
define  $\|w\|_{\infty} = \sup_{1 \le i \le |w|, i \in \mathbb{N}} \|w[i]\|_{\infty}$.}
For a  family $\{A_x\}_{x \in X}$ of $n \times n$ matrices 
and a sequence
$w  \in X^*$ define 
\[ A_{w} = A_{w[|w|]} \cdots A_{w[1]}, ~~\textcolor{black}{\mbox{ if } w \in X^{+}}, \quad A_{\epsilon}=I_n. \]


\subsection{State-space representations}
   Consider a general state-space representation (\emph{SSR})
\begin{equation}
\label{gen:eq1}
\Sigma \left \{\begin{split}
 x(t+1)&=f(x(t),u(t)), x(0) \in B \\
 y(t)& =h(x(t),u(t))
\end{split}\right. \textcolor{black}{,}
\end{equation}
where $x(t) \in \mathbb{R}^n$ is the state, $u(t) \in \mathbb{U}$ is the input, $y(t) \in \mathbb{R}^{\textcolor{black}{n_y}}$ is the output,
$f:\mathbb{R}^n \times \mathbb{U} \rightarrow \mathbb{R}^n$, 
$h:\mathbb{R}^n  \times \mathbb{U} \rightarrow \mathbb{R}^{\textcolor{black}{n_y}}$ is the \emph{state-transition} and
\emph{readout} map respectively, and  
   $B=\{B_j \in \mathbb{R}^n \}_{j \in J}$, \textcolor{black}{$J \ne \emptyset$}, is the family of \emph{initial states}. 
We identify  $\Sigma$ with the tuple $(f,B,h)$ and we denote by  $\dim(\Sigma)$ the dimension $n$ of its state-space.
%

For an initial state $x_0$,  and sequence $w \in \mathbb{U}^{*}$
denote by $x_{\Sigma}(x_0,w)$ the state of \eqref{gen:eq1} 
reached at time $t$ under input $u(0)=w[1],\ldots, u(t-1)=w[t]$, $t=|w|$ from \textcolor{red}{the} initial state $x_0$, and we set 
$x_{\Sigma}(x_0,\epsilon)=x_0$.
Define the response map $y_{\Sigma,x_0}\textcolor{black}{:\mathbb{U}^{+} \rightarrow \mathbb{R}^{n_y}}$ of $\Sigma$ induced by the
initial state $x_0$ as follows: \textcolor{black}{$y_{\Sigma,x_0}(v\sigma)=h(x_{\Sigma}(x_0,v),\sigma)$,
$v \in \mathbb{U}^{*},\sigma \in \mathbb{U}$, i.e.,
  $y_{\Sigma,x_0}(u(0)\cdots u(t))$ is the output response of
  \eqref{gen:eq1} at time $t$ to the inputs $\{u(s)\}_{s=0}^{t}$ from the initial state $x_0$.
}
We say that $\Sigma$ is \emph{span-reachable} \textcolor{black}{(respectively \emph{affine-reachable})}, 
if the linear (respectively affine) span of all the reachable states $\{x_{\Sigma}(B_j,w)\}_{w \in \mathbb{U}^{*}, B_j \in B}$
equals $\mathbb{R}^n$. We say that $\Sigma$ is \emph{observable}, if 
for any two distinct initial states $x_0 \ne \bar{x}_0$ there exists $w \in \mathbb{U}^{+}$ such that
$y_{\Sigma,x_0}(w) \ne y_{\Sigma,\bar{x}_0}(w)$. 


\emph{State-linear systems (SLSs)} are special cases of SSR \eqref{gen:eq1} 
for which $f(x,u)=A_u x$ and $h(x,u)=Cx$. In what follows, we identify $\Sigma$ with the tuple 
$\Sigma=(\{A_u\}_{u \in \mathbb{U}},B,C)$. 
   
\emph{State-affine systems (SAS) \cite{Son:Real}} are SSRs \eqref{gen:eq1} of the form
\begin{equation}
\label{eq1:sas}
\begin{split}
 x(t+1)&= A_{u(t)}x(t)+b_{u(t)}, \quad  x(0) \in B=\{x_0\}, \\
 y(t) &= Cx(t)+d,
\end{split}
\end{equation}
where $A_u \in \mathbb{R}^{n \times n}$, $b_u \in \mathbb{R}^{n}$ for $u \in \mathbb{U}$, and
$C \in \mathbb{R}^{\textcolor{black}{n_y} \times n}$, $d \in \mathbb{R}^{\textcolor{black}{n_y}}$, and the family of initial states $B$ is a singleton.
Following \cite{Son:Real} we define the subclass of bounded  SASs, for 
which realization theory is computationally effective.  To this end, 
let $0 < n_p \in \mathbb{N}$ 
and consider $n_p+1$
linearly independent  functions $\psi_{q}:\mathbb{U} \rightarrow \mathbb{R}$, 
$q=0,\ldots,n_p$. 
Then  SSR of \eqref{eq1:sas} is \emph{$\{\psi_q\}_{q=0}^{n_p}$-bounded SAS}, if 
$A_ux+b_u\!\!=\!\!\sum_{i=0}^{n_p} (A_ix+b_i) \psi_i(u)$ and  $h(x,u)\!\!=\!\!Cx+d$
for \textcolor{black}{some} matrices and vectors 
$\{A_q,b_q\}_{q=0}^{n_p}$.
\emph{Bilinear systems \cite{isi:tac}} are bounded SASs 
\textcolor{black}{with $\mathbb{U}=\mathbb{R}^{m}$, $\psi_0=1$, $\psi_q(u)=u_q$ for $q \in \{1,\ldots,m\}$, $m=n_p$, $b_0=0$, $d_0=0$.}

%
 
 A \emph{linear parameter-varying (LPV) system with affine dependence on the parameters (LPV-SSA) \cite{Toth1}} is an SSR
 with inputs $u(t)=(p(t),v(t)) \in \mathbb{U}=\mathbb{P} \times \mathbb{R}^m$ of the form
\begin{equation}
\label{eq1:lpv}
\begin{split}
x(t+1) & = A(p(t)) x(t) + B(p(t)) v(t), x(0)=0\\
y(t) &= C(p(t))x(t) \textcolor{black}{,}
\end{split}
\end{equation}
where $p$ is called the \emph{scheduling signal}, $v$ is the control input, and 
the matrices $A(p(t))$, $B(p(t))$, and $C(p(t))$ are affine in $p(t)$, i.e.
\(
   X(p(t)) = X_0 + \sum_{i=1}^{n_p} X_i p_i(t)
\)
with $X = A, B, C$ and $\{A_i, B_i, C_i\}_{i=1}^{n_p}$ being suitable matrices,
\textcolor{black}{and the family of initial states is the singleton $\{0\}$}.
We assume that $\mathbb{P} \subseteq \mathbb{R}^{n_p}$ is contained in the standard simplex $\Delta = \{x = (x_1, \ldots, x_{n_p})^T \in [0,1]^{n_p} \mid \sum_i x_i = 1\}$
and $\mathbb{P}$ contains the standard unit vectors $e_1, \ldots, e_{n_p}$.
We also assume $A_0 = 0$, $B_0 = 0$, $C_0 = 0$, otherwise we can replace $A_i, B_i, C_i$ by $A_i + A_0, B_i + B_0, C_i + C_0$ for $i = 1, \ldots, n_p$.
A \emph{linear switched system (LSS) \cite{Sun:Book,D:Lib}} is a special case of LPV-SSA where $\mathbb{P} = \{e_1, \ldots, e_{n_p}\}$,
 \textcolor{black}{and  $X(e_i) = X_i$ for $X = A, B, C$ and $i = 1, \ldots, n_p$, and the scheduling signal $p$ is identified with the
  \emph{switching signal} $q$, where $p(t) = e_{q(t)}$ for all $t \in \mathbb{N}$.  }
Assume that both \textcolor{black}{$\widetilde{\Sigma}=(\widetilde{f},\widetilde{B},\widetilde{h})$} and $\Sigma$ belong to one of the classes: SLS, SASs,LPV-SSA, and LSS. Then $\widetilde{\Sigma}$ and $\Sigma$ \textcolor{black}{are called} \emph{isomorphic} if 
\textcolor{black}{$n\!\!=\!\!\dim(\Sigma)\!\!=\!\!\dim(\widetilde{\Sigma})$} and there exists an invertible map $T$ on $\mathbb{R}^n$ 
which is linear for SLSs, LSSs, LPV-SSAs,  and affine for SASs, such that
            \begin{equation*}
                 \begin{array}{rcl}
                 \label{repr:morph:eq2}
                 &Tf(x,u)=\widetilde{f}(Tx,u), ~ TB_j=\widetilde{B}_j,  ~h(x,u)=\widetilde{h}(Tx,u) \textcolor{black}{,}
             \end{array}
            \end{equation*}
            \textcolor{black}{for all $x \in \mathbb{R}^n$, $u \in \mathbb{U}$ and $j \in J$.}
            
If the readout map $h$ in \eqref{gen:eq1} does not depend on $u$, then we call $\Sigma$ a \emph{state-output system (SO)}. In this case, the output at time $t$ depends only on the inputs up to $t-1$.
We call functions $H:\mathbb{U}^{+}\to\mathbb{R}^{\textcolor{black}{n_y}}$ \emph{response maps}, and functions $H:\mathbb{U}^{*}\to\mathbb{R}^{\textcolor{black}{n_y}}$ \emph{SO response maps}. Intuitively, $H(u_0\cdots u_t)$ is the output at time $t$ (resp.\ at $t+1$ if $H$ is SO) generated by inputs $u_0,\ldots,u_t$, and for SO maps $H(\epsilon)$ is the output at time $0$. 
We model the input--output behavior of (SO) SSRs as the following families 
of (SO) response maps: 
\begin{align}
    \Psi=\{H_j:\mathbb{U}^{+}\rightarrow \mathbb{R}^{\textcolor{black}{n_y}}\}_{j \in J} & \mbox{ if general SSR} 
    \label{family:resp:gen}
    \\
    \Psi=\{ H_j:\mathbb{U}^{*}\rightarrow \mathbb{R}^{\textcolor{black}{n_y}}\}_{j \in J} & \mbox{ if SO SSR}  \textcolor{black}{.}
    \label{family:resp}
\end{align}
The SSR $\Sigma$ is called a \emph{realization} of $\Psi$, if 
\[
    H_j(w)=\left\{\begin{array}{rl}
             y_{\Sigma,B_j}(w) & \mbox{if $\Psi$ is from \eqref{family:resp:gen} and $w\in\mathbb{U}^{+}$} \\
	     h(x_{\Sigma}(B_j,w)) & \mbox{if $\Psi$ is from \eqref{family:resp} and $w \in \mathbb{U}^{*}$} \\
	                           & \textcolor{red}{\mbox{ and $\Sigma$ is SO}}
             \end{array}\right. 
\]
for all $j \in J$. 
If $J$ is a singleton, we identify $\Psi=\{H\}$ with $H$ \textcolor{black}{and call} any SSR realization of $\Psi$  a realization of $H$.
Let $\mathscr{C}$ be a subclass of SSRs (e.g., SLS, SAS, LSS, LPV-SSA).
We call $\Sigma\in\mathscr{C}$ a \emph{minimal $\mathscr{C}$-realization} of $\Psi$ if \textcolor{black}{its dimension $\dim(\Sigma)$} is minimal 
among all $\mathscr{C}$-realizations of $\Psi$.


%



\subsection{Realization theory of state-linear systems}
\label{sect:real}
  Since the readout map of a SLS does not depend on the last input, we model the input-output behavior of SLSs as families of SO response maps
  \eqref{family:resp}. Then SLS \eqref{eq1} realizes $\Psi$, if and only if 
  \(
  H_j(w)=CA_wB_j,
  \), for all
  $w\in\mathbb{U}^*,~ j\in J$.
  We recall the necessary and 
  sufficient conditions for $\Psi$ being realizable by a SLS.
  To this end, define the shift of a response map $H:\mathbb{U}^{*} \rightarrow \mathbb{R}^{\textcolor{black}{n_y}}$ by sequence $v \in \mathbb{U}^{*}$ as follows:
  \[ v \circ H: \mathbb{U}^{*} \ni w \mapsto H(vw). \]
 We define the \emph{Hankel-space} $W_\Psi$ as the linear space generated by the shifted versions of elements of $\Psi$
 \[ 
     W_{\Psi}=\mathrm{Span} \{v \circ H_j \mid j \in J, v \in \mathbb{U}^{*} \} \textcolor{black}{,}
 \]
 and we call $\dim W_{\Psi}$ the \emph{Hankel-rank}. 
 \textcolor{black}{Here, for response maps addition and multiplication by a scalar
 is understood point-wise: for two SO response maps $H_1,H_2$, for scalars $\alpha,\beta \in \mathbb{R}$, $(\alpha H_1+\beta H_2)(w)=\alpha H_1(w)+\beta H_2(w)$, $w \in \mathbb{U}^{*}$.}
 \textcolor{black}{If $\mathbb{U}$ and $J=\{j_1,\ldots,j_d\}$ 
 are both
 finite, then the}  SO response map $H_j$ corresponds to formal power series,
 and $W_{\Psi}$ is isomorphic to the column space of the Hankel-matrix
 \cite{MFliessHank,Reut:Book,Son:Real}
 of
 the formal power series $S:\mathbb{U}^{*} \ni v \mapsto \begin{bmatrix} H_{j_1}(v) & \ldots&  H_{j_d}(v) \end{bmatrix}$.
\begin{Theorem}[Theorem 2.15 \cite{Son:Real}]
\label{theo:real1}
      The family $\Psi$   
      has a realization by a SLS, if
      and only if $\dim W_{\Psi}=n <+\infty$.
      All minimal realizations of $\Psi$ are of dimension $n=\dim W_{\Psi}$, and 
      they are isomorphic.
      A SLS $\Sigma$ is a minimal realization of $\Psi$, if and only if 
      $\Sigma$ is span-reachable and observable. 
\end{Theorem}

\textcolor{black}{It is well-known that the result above can be made computationally effective when $\mathbb{U}$ and $J$ are finite \cite{MFliessHank,Reut:Book}.}
For what follows we need the following technical result.
\begin{Lemma}
  \label{theo:real2}
  If $\Sigma$ is a minimal SLS realization of $\Psi$, then
  there exist sequences $\{v_i,w_i\}_{i=1}^{n} \subseteq \mathbb{U}^{*}$,
  and indices $\{j_i\}_{i=1}^{n} \subseteq  J$
  such that the matrices $O$ and $R$ have rank $n$:
  \begin{align*}
       R=\begin{bmatrix} A_{v_1}B_{j_1} & \ldots & A_{v_n}B_{j_n} \end{bmatrix} ,\\
      O=\begin{bmatrix} A_{w_1}^TC^T & \ldots & A_{w_n}^TC^T \end{bmatrix}^T\textcolor{black}{.}
  \end{align*}
\end{Lemma}
\begin{proof}
  \begin{color}{black}
  For finite $\mathbb{U}$ the statement is well-known \cite{MFliessHank, Reut:Book,Son:Real}, the proof for the general case is the same, as sketched below. 
 Since $x_{\Sigma}(B_j,w)=A_wB_j$, 
 span-reachability of $\Sigma$ is equivalent to the set $\{A_wB_j \}_{j \in J, w \in \mathbb{U}^{*}}$ containing a basis 
 of $\mathbb{R}^n$.
 By choosing $\{v_i,j_i\}_{i=1}^{n}$ so that $\{A_{v_i}B_{j_i}\}_{i=1}^{n}$ is such a  basis, 
 the first statement follows. As $y_{\Sigma,x_0}(v\sigma)=CA_vx_0$, 
 $v \in \mathbb{U}^{*}, \sigma \in \mathbb{U}$, 
 it follows that
 $y_{\Sigma,x_1}(v\sigma)=y_{\Sigma,x_2}(v\sigma)$ if and only if $x_1-x_2 \in  \ker CA_v$.
 Hence 
 observability of $\Sigma$ is equivalent to $\bigcap_{w \in \mathbb{U}^{*}} \ker CA_w=\{0\}$, i.e., the set of rows of
 $\{CA_w\}_{w \in \mathbb{U}^{*}}$ 
 contains a basis of $\mathbb{R}^{1 \times n}$. By choosing $\{w_i\}_{i=1}^{n}$ so that $\{C_{r_i,\cdot}A_{w_i}\}_{i=1}^{n}$ is such a basis for some output indices $\{r_i\}_{i=1}^{n} \subseteq \{1,\ldots,n_y\}$,
 the second statement follows. 
  \end{color}
\end{proof}

\section{Main result}
We call the SLS  \eqref{eq1} \emph{globally uniformly  exponentially stable (GUES)
with decay rate $\beta$}, if \textcolor{black}{the set of initial states $B=\{B_j\}_{j \in J}$ is bounded} and
  there exists $C_g \ge 1$ such that 
    \begin{equation}
    \label{def:sls:gues}
    \textcolor{black}{\forall w \in \mathbb{U}^{*}: }\quad 
    \|A_w\|_2 \le C_g\beta^{|w|}. 
    \end{equation}
Note \textcolor{black}{that} \eqref{def:sls:gues} \textcolor{black}{implies that} any solution of \eqref{eq1} \textcolor{red}{satisfies} $\|x(t)\|_2 < C\beta^t \|x(0)\|_2$ \cite{RJungers}. 
    The family $\Psi$ from \eqref{family:resp} 
     is called \emph{uniformly decaying} (UD), if 
     $\sup_{\substack{w \in \textcolor{black}{\mathbb{U}^{*}}, \\ j \in J}} \|H_j(w)\|_{\textcolor{black}{2}} < +\infty$, and 
     \begin{equation}
     \label{def:ud}
     \lim_{t \rightarrow \infty} \sup_{w \in \mathbb{U}^{*}, |w|=t,  j \in J} \|H_j(w)\|_{\textcolor{black}{2}}=0 
     \end{equation}
     and $\Psi$ is called  \emph{uniformly exponentially decaying (UED) with decay rate $\beta \in (0,1)$}, 
     if there exists $C_{ed} \ge 1$, 
     such that 
     \begin{equation}
     \label{def:ued}
     \forall w \in \mathbb{U}^*: ~ \sup_{j \in J} \|H_j(w)\|_{\textcolor{black}{2}} \le C_{ed}\beta^{|w|}.
     \end{equation}
     Clearly, UED implies UD.
     Now, we relate realizability by a stable SLS to these properties.
  \begin{Theorem}
  \label{main:theo}
  For a family $\Psi$ as in \eqref{family:resp}, the following hold.
  \\
    \textbf{(A)}
    $\Psi$ can be realized by a SLS which is GUES 
    if and only if $\Psi$ is UD  and $\Psi$ has a finite Hankel-rank.
    \\
    \textbf{(B)}
    If $\Psi$ is UED with rate $\beta$, then all minimal SLS realizations of $\Psi$ are GUES with rate $\beta$.  
    \\
  \textbf{(C)}
      If $\Psi$ has a finite Hankel-rank and $\Psi$ is UD, then $\Psi$ is also UED.     
  \end{Theorem}
  \begin{proof}
     Let $\Sigma$ \eqref{eq1}
     be a realization of $\Psi$.
     The theorem follows from the subsequent implications.
     
    \textbf{(I) $\Sigma$ is GUES with rate $\beta$ $\implies$ \eqref{def:ued} holds.}
   Since 
   $\|H_j(w)\|_2=\|CA_wB_j\|_2
   \le \|C\|_2 \|A_w\|_2\|B_j\|_2$,
   if
   \eqref{def:sls:gues} holds, then
    with $C_{ed}=C_g \|C\|_2 \sup_{j \in J} \|B_j\|_2$ \eqref{def:ued}
    holds.
    
  \textbf{\textcolor{red}{(II)} $\Psi$ is UD $\implies$ all minimal SLSs are GUES.}
    \textcolor{red}{Let} $\Sigma$ be  a minimal SLS which realizes $\Psi$.
    Then the matrix $O$ from Lemma \ref{theo:real2} has a left inverse $O^{+}$, and 
    the matrix $R$ from Lemma \ref{theo:real2} is invertible.
    \textcolor{red}{For} any $\underline{u} \in \mathbb{U}^{*}$,
     \( CA_{w_l}A_{\underline{u}}A_{v_i}B_{j_i} \!\!=\!\! H_{j_i}(v_i\underline{u}w_l) \), 
    hence 
    \begin{equation*}
      \begin{split}
      & \|OA_{\underline{u}}R\|_2 
      \le \sum_{i,l=1}^{n} \|CA_{w_l}A_{\underline{u}}A_{v_i}B_{j_i}\|_2\!\!= \!\!
         \sum_{i,l=1}^{n}  \|H_{j_i}(v_i\underline{u}w_l)\|_2 
      \end{split}  
    \end{equation*}
    As
    $\|A_{\underline{u}}\|_2\!\!=\!\!\|O^{+}OA_{\underline{u}} \textcolor{black}{R} R^{-1}\|_2 \le  \|O^{+}\|_2 \|OA_{\underline{u}}R\|_2 \|R^{-1}\|_2$,
    \begin{equation}
    \label{main:theo:pf4}
    \begin{split}
     & \|A_{\underline{u}}\|_2
     \le \|O^{+}\|_2 \|R^{-1}\|_2 
        \sum_{i,l=1}^{n}  \|H_{j_i}(v_i\underline{u}w_l)\|_2.
    \end{split}
  \end{equation}
  If $\Psi$ is UD, then by setting $S=\sup_{w \in \mathbb{U}^{*}, j \in J} \|H_j(w)\|_2$
  it follows from \eqref{main:theo:pf4} that 
  \textcolor{black}{$\|A_{\underline{u}}\|_2 \le \|O^{+}\|_2 \|R^{-1}\|_2 n^2S$} 
  i.e., the set $\{ A_{u} \}_{u \in \mathbb{U}}$ is bounded. 
   \textcolor{black}{From} \eqref{def:ud}
   it follows that
    for every $\epsilon > 0$ there exists $N_{\epsilon}$ such that
    if $|\underline{u}| > N_{\epsilon}$, \textcolor{black}{then} 
    \[  \begin{split}
       \|H_{j_i}(v_i\underline{u}w_l)\|_2 < \frac{\epsilon}{n^2 \|O^{+}\|_2 \|R^{-1}\|_2 }
      \end{split}.
    \]
  Hence, by using \eqref{main:theo:pf4} it follows that 
  \( \|A_{\underline{u}}\|_2 \le \epsilon \)
  \textcolor{black}{if} $|\underline{u}| > N_{\epsilon}$, i.e.,
  \eqref{eq1} is globally uniformly asymptotically stable. 
  It is well-known that then \eqref{eq1} is \textcolor{black}{GUES}
  \cite[Theorem 2.15]{Sun:Book}, \cite{Piatnitsky}.
  Indeed, as $\{ A_u\}_{u \in \mathbb{U}}$ is bounded, by
  \cite[Corollary 1.1]{RJungers} 
  the joint spectral radius
  $\rho=\lim_{t \rightarrow \infty}\sup_{\underline{u} \in \mathbb{U}^t} \|A_{\underline{u}}\|_2^{1/t}$ of   \textcolor{black}{$\{A_u\}_{u \in \mathbb{U}}$ satisfies $\rho \!\!<\!\! 1$.}
  Let $\epsilon > 0$ be such that $\beta=(\rho+\epsilon)<1$.
  Then there exists $T_{\epsilon}$ such that if $|\underline{u}| > T_{\epsilon}$
  then $\|A_{\underline{u}}\|_2 < \beta^{|\underline{u}|}$.
  \textcolor{black}{For \textcolor{black}{$|\underline{u}| \le T_{\epsilon}$},
  $\|A_{\underline{u}}\|_2 \le \textcolor{black}{n^2S\|O^{+}\|_2 \|R^{-1}\|_2} \textcolor{black}{\le n^2S\|O^{+}\|_2 \|R^{-1}\|_2 \beta^{|\underline{u}|-T_{\epsilon}}}$.
  Hence, 
  \eqref{def:sls:gues} holds with 
  $C_g\!\!=\!\!\max\{1, \textcolor{black}{n^2S\|O^{+}\|_2 \|R^{-1}\|_2 \beta^{-T_{\epsilon}}}\}$.}
    \textcolor{black}{Also},
   \(
   \|OB_j\|_2 \le 
     \sum_{i=1}^{n} \|CA_{w_i}B_j\|_2=
      \sum_{i=1}^{n} \|H_j(w_i)\|_2 \le \textcolor{black}{nS}
   \)
   and hence
   \( \|B_j\|_2=\|O^{+}O B_j\|_2 \le \|O^{+}\|_2 \textcolor{black}{nS}  \).  
   
\textbf{(III) If $\Sigma$ is minimal and \eqref{def:ued} holds, then \eqref{def:sls:gues} holds with the same $\beta$}. 
  It follows from \eqref{def:ued} that
  \( \|H_{j_i}(v_i\underline{u}w_l)\|_2 <  C_{ed}\beta^{|v_i|+|\underline{u}|+|w_l|} \).
  Using \eqref{main:theo:pf4}, 
  and choosing \textcolor{black}{$C_g=\|O^{+}\|_2\|R^{-1}\|_2\max\{1,C_{ed}\sum_{l,i=1}^{n} \beta^{|v_i|+|w_l|}\}$},  it follows that
  \eqref{def:sls:gues} holds. 
  Statement \textbf{(A)} follows from \textbf{(I)}, \textbf{(II)}, and the fact that finite Hankel-rank guarantees the existence of a minimal SLS realization. Statement \textbf{(B)} follows from \textbf{(III)}.  Statement \textbf{(C)} follows from \textbf{(II)} and \textbf{(I)}.
  \end{proof}

\section{Special cases} 

In this section we apply the main results to  SASs, LPV-SSAs and LSSs, by associating any particular system of those classes with a suitable SLS following \cite{Son:Real},\cite{Petreczky2016}.

\subsection{State-affine systems}
  We consider SASs of the form \eqref{eq1:sas}. These systems are SO, and their 
  \textcolor{black}{response maps are of the form} $H:\mathbb{U}^{*} \rightarrow \mathbb{R}^{\textcolor{black}{n_y}}$.
  Following \cite{Son:Real}, we can reduce realization theory of SASs to realization theory of SLSs as follows. 
   We associate with $H$ the 
   family $\Psi_H=\Psi$ of the form
   \eqref{family:resp}, where 
   $J=\mathbb{U}$ and 
   \begin{equation}
   \label{def:hu}
     \forall w \in \mathbb{U}^{*}, u \in \mathbb{U}: \quad H_u(w)=H(uw)-H(w) \textcolor{black}{.}
   \end{equation}
    Clearly, 
    $H(w)\!=\!\textcolor{black}{H(\epsilon)+\sum_{s=1}^{|w|}} H_{w[s]}(w[s+1]\cdots w[|w|])$, \textcolor{black}{and}
    \eqref{eq1:sas} is a realization of $H$, if and only if for all 
    $v \in \mathbb{U}^{*}$
    \[ 
    H_u(v)=CA_v (A_ux_0+b_u-x_0),  ~ u \in \mathbb{U}, ~ H(\epsilon)=Cx_0+d. 
    \]
    Following \cite{Son:Real}, if
    $\Sigma$ is as in \eqref{eq1:sas}, we  define
    the SLS
    \[
    R_{\Sigma}=\bigl(\textcolor{black}{\{A_u\}_{u\in\mathbb{U}}},B=\{B_u:=A_ux_0+b_u-x_0\}_{u\in\mathbb{U}},C)\textcolor{black}{.}
    \]
    Then $\Sigma$ is a realization of $H$
    if and only if $R_{\Sigma}$ is a SLS realization of $\Psi_H$
    and $Cx_0+d=H(\epsilon)$.
    \textcolor{black}{The} transformation
    $\Sigma \mapsto R_{\Sigma}$ preserves observability and isomorphism, and maps affine-reachable SASs to span-reachable SLSs. 
    Conversely, if
    $R=(\{A_{u}\}_{u \in \mathbb{U}},B,C)$
    is a SLS realization of
    $\Psi_H$, then 
    the SAS $\Sigma_R$ of the form
    \eqref{eq1:sas} with $d=H(\epsilon)$, $x_0=0$ and 
     $b_u=B_u$ is a realization of $H$.
     The transformation $R \mapsto \Sigma_R$ preserves 
     span-reachability, observability, minimality and isomorphism.
     For $\Sigma_R$ affine- and span-reachability are equivalent, as its initial state is zero. 
%
     
     From Theorem \ref{theo:real1} it thus follows \cite{Son:Real}
     that  \textbf{(1)} $H$ has a realization by a SAS if and only 
     if $\Psi_H$ has a finite Hankel-rank, \textbf{(2)} a SAS is a minimal realization of $H$ 
     if and only if it is affine-reachable and observable, \textbf{(3)} all minimal SAS realizations of $H$ are isomorphic. 

     We  call \eqref{eq1:sas} \emph{globally uniformly exponentially stable (GUES) with decay rate $\beta$} if 
     the associated SLS is GUES. 
      This definition does not require \eqref{eq1:sas} to be stable in the classical sense as an equilibrium point may not exist. However, if $\mathbb{U}$ is finite and \eqref{eq1:sas} is GUES, then by \cite{ATHANASOPOULOS2016158,Rossa2023} it has an exponentially attractive minimal invariant set, which consists of the equilibrium point if the latter exists.
     
     \textcolor{black}{For the response 
     maps of SASs, the counterpart  of the decay property is the  forgetting property. }
      We say that $H$ is \emph{uniformly forgetting (UF)}
      \textcolor{black}{if $\sup_{v \in \mathbb{U}^{*},u \in \mathbb{U}}\|H(uv)-H(v)\|_2 < +\infty$ and}
       \begin{equation}
       \label{def:uf}
       \lim_{t \rightarrow \infty} \sup_{\textcolor{black}{v \in \mathbb{U}^{+}, |v|=t}, w \in \mathbb{U}^{*}} \|H(wv)-H(v)\|_{\textcolor{black}{2}}=0\textcolor{black}{,}
       \end{equation}
       and we say that $H$ is \emph{uniformly exponentially forgetting (UEF)} with rate $\beta \in (0,1)$, if there exists $C_{f} > 0$, 
       \begin{equation}
       \label{def:uef}
       \forall \textcolor{black}{v \in \mathbb{U}^{*}}:  \sup_{w \in \mathbb{U}^{*}} \|H(wv)-H(v)\|_{\textcolor{black}{2}} \le C_f \beta^{|v|}\textcolor{black}{.} 
       \end{equation}
      \textcolor{black}{(Exponential)} forgetting  means that the influence of earlier inputs on the current output \textcolor{black}{(exponentially)} fades with time, and it 
      \textcolor{black}{relates to decaying properties of $\Psi_H$ as follows.}
  \begin{Lemma}
      \label{UD:UF}
      \textcolor{black}{If $\Psi_H$ is UED with rate $\beta$ if and only if $H$ is  UEF with rate $\beta$}. Moreover, if $H$ is UF, then $\Psi_H$ is UD.
      \end{Lemma}
    \begin{proof}
             If $\Psi_H$ is UED with rate $\beta$, then 
             \textcolor{black}{by using
             \( \|H(wv)-H(v)\|_2 \le  \textcolor{black}{\sum_{i=1}^{|w|}} \|H_{w[i]}(w[i+1] \cdots w[|w|]v)\|_2 \), 
             \textcolor{black}{$v \in \mathbb{U}^{*}$}, 
             and
              bounding the terms in the sum by $C_{ed}\beta^{|v|+|w|-i}$,}
             we get 
             \eqref{def:uef} \textcolor{black}{with $C_f=C_{ed}/(1-\beta)$. 
             Conversely, \eqref{def:uef} implies $\|H_{u}(v)\|_2=\|H(uv)-H(v)\|_2 \le C_f\beta^{|v|}$, $v \in \mathbb{U}^{*}$, i.e., $\Psi_H$ is UED with rate $\beta$.}
            %
            If $H$ is UF, then
            \textcolor{black}{from \eqref{def:hu} it follows that}
            $\sup_{v \in \mathbb{U}^{*}, |v|=t} \|H_u(v)\|_2 \le
            \sup_{\substack{v \in \textcolor{black}{\mathbb{U}^{+}}, |v|=t \\ w \in \mathbb{U}^{*}}}
                \|H(wv)-H(v)\|_2$
                \textcolor{black}{for $t \ge 1$. Hence,}
                \textcolor{black}{by} taking the limit $t \rightarrow \infty$ we get \eqref{def:ud}, \textcolor{black}{and 
                 $\sup_{\substack{u \in \mathbb{U} \\ v \in \mathbb{U}^{*}}} \|H_u(v)\|_2\!=\! \sup_{\substack{u \in \mathbb{U} \\ v \in \mathbb{U}^{*}}}\! \! \|H(uv)-H(v)\|_2<+\infty$.}
    \end{proof}
   From Theorem \ref{main:theo} we get thus  the following.
   \begin{Corollary}
   \label{col:sas}
      The response map 
        $H$ can be realized by a GUES SAS 
        if and only if $H$ is UF  and 
        $\Psi_H$ has a finite Hankel-rank.
        If $\Psi_H$ has a finite Hankel-rank, $H$ is UF if and only if $H$ is UEF.
	   If $H$ is UEF with \textcolor{red}{rate} $\beta$, then all minimal SAS realizations of $H$ are GUES with rate $\beta$. 
         \end{Corollary}
         \begin{color}{black}
       \begin{Remark}[Echo-state, fading memory]
        \label{rem:echo_fading}
      The UF property is related to the \emph{fading-memory (FM)} property for maps on infinite input sequences \cite{Ortega3, Ortega4}. If $H$ is UF and realizable by a SAS, then any minimal realization \eqref{eq1:sas} of $H$ is GUES; thus, by \cite[Theorem 1]{Ortega4}, it has the echo-state property and induces the infinite-sequence response map
      \( H_{ext}(\underline{u}) = Cb_{\underline{u}[0]} + \sum_{k=-1}^{-\infty} CA_{\underline{u}[0]\cdots \underline{u}[-k+1]}b_{\underline{u}[-k]} = 
      \lim_{t \to -\infty} H(\underline{u}[-t] \cdots \underline{u}[0]) \),
      which is well-defined and has the FM property. We conjecture that FM maps on infinite sequences are those realizable by GUES SASs, and that they induce UF maps on finite sequences.
      \end{Remark}  
    \end{color}
    Among SASs $\psi$-bounded SASs are of particular interest, as they include \textcolor{black}{such popular classes as} bilinear systems, and they are computationally effective.
  By \cite{Son:Real}, $H$ can be realized by a $\psi$-bounded SAS if and only if $H$ is bounded of type $\{\psi_{q}\}_{q=0}^{n_p}$ according to \textcolor{black}{\cite[Definition 1.15]{Son:Real}}, and the Hankel-rank of $\Psi_H$ equals the rank of a suitable Hankel matrix. \textcolor{black}{Minimization of  $\psi$-bounded SASs results in $\psi$-bounded SASs, which} are minimal SASs realizations. Thus:
  \begin{Corollary}
   The statements of Corollary \ref{col:sas} remain true if  SAS is replaced by $\psi$-bounded SAS and the finite Hankel-rank of $\Psi_H$ is replaced by the finite rank of the Hankel-matrix defined in \cite[Definition 2.7]{Son:Real}. 
  \end{Corollary}

\subsection{LPV-SSA/LSS} 
\textcolor{black}{We} apply Theorem \ref{main:theo} to LPV-SSA/LSSs.
Their 
response maps 
are of the form $H: (\mathbb{P} \times \mathbb{R}^m)^+ \to \mathbb{R}^{\textcolor{black}{n_y}}$.
Following \cite{Petreczky2016}, we  say that $H$ has an \textcolor{black}{\emph{impulse response representation (IRR)}} if there exists a map $S: \textcolor{black}{\bigcup_{k=2}^{\infty} Q^{k}} \to \mathbb{R}^{\textcolor{black}{n_y} \times m}$, 
where $Q=\{1,\ldots,n_p\}$, such that 
\textcolor{black}{for any $\underline{p} \in \mathbb{P}^{\infty}$, $\underline{v} \in (\mathbb{R}^m)^{\infty}$}, 
\begin{align*}
 & H((\underline{p}[1], \underline{v}[1])\cdots (\underline{p}[t], \underline{v}[t])) = \sum_{\tau=1}^{\textcolor{black}{t-1}} (h \diamond \underline{p})(t,\tau)\underline{v}[\tau], \quad \textcolor{black}{t \ge 1} \\
   & (h \diamond \underline{p})(t,\tau)=\sum_{s \in Q^{t-\tau+1}} S(s)\underline{p}_s(t), \quad \textcolor{black}{\tau \in [1,t-1]}
\end{align*}
\textcolor{black}{and $\underline{p}_s(t)=\prod_{i=1}^{|s|} \textcolor{black}{(\underline{p}[t-|s|+i])_{s[i]}}$\textcolor{black}{; recall that $(\underline{p}[k])_r$ is the $r$th component of $\underline{p}[k]$}.}
  We call $h \diamond \underline{p}$ the \textcolor{black}{\emph{impulse response} of $H$ (IRR)}  for the scheduling sequence $\underline{p}$\textcolor{black}{, 
  and we call the values of $S$ \emph{sub-Markov parameters.}}
  \textcolor{black}{By \cite{Petreczky2016}, 
  $\Sigma$ from \eqref{eq1:lpv} is a realization of $H$
  if and only if $H$ has an IRR and $S(q_1vq_2)\!=\!C_{q_2}A_vB_{q_1}$, 
  for all $q_1,q_2 \in Q$, $v \in Q^{*}$.}
  If $\mathbb{P}=\{e_1,\ldots,e_{n_p}\}$ (switched case) and \textcolor{black}{$q \in Q^{\infty}$ is such that}
  $\underline{p}[s]=e_{\underline{q}[s]}$, then
  $(h \diamond \underline{p})(t,\tau)=S(\underline{q}[\tau]\cdots \underline{q}[t])$, i.e., the sub-Markov parameters are the impulse responses.
  For general $\mathbb{P}\subseteq\Delta$, $(h \diamond \underline{p})(t,\tau)$ is a convex combination of these parameters.
  Following \cite{Petreczky2016} we define the \textcolor{black}{maps} $M_{q,\textcolor{black}{l}}: Q^* \to \mathbb{R}^{n_p \textcolor{black}{n_y}}$,
  $q \in Q$, $l=1,\ldots,m$ \textcolor{black}{as} 
   \[
   M_{q,l}(w) = \begin{bmatrix} (S_{\cdot,\textcolor{black}{l}}(qw1))^T, & \ldots, & (S_{\cdot,\textcolor{black}{l}}(qwn_p))^T \end{bmatrix}^T.
  \]
  \textcolor{black}{
  for all 
  $w \in Q^{*}$, 
  where $S_{\cdot,\textcolor{black}{l}}(s)$ is the \textcolor{black}{$l$}th column of $S(s)$,}
  and we consider the following family of response maps:
  \textcolor{black}{$\Psi_H = \{M_{q,l}\}_{j=(q,l) \in J}$ with $J=Q \times \{1,\ldots,m\}$.
   For a LPV-SSA/LSS $\Sigma$ from \eqref{eq1:lpv}, for any $(q,l) \in J$
   let $\tilde{B}_{(q,l)}$ be the $l$th column of $B_q$, and 
   associate with $\Sigma$ the SLS
  }
  \[ 
  \begin{split}
  &R_{\Sigma}=(\{A_q\}_{q \in Q},\textcolor{black}{\{\tilde{B}_{(q,l)}\}_{(q,l) \in J}},\begin{bmatrix} C_1^T, & \ldots & C_{n_p}^T \end{bmatrix}^T). \\
  \end{split}
  \]
  \textcolor{black}{Then} $\Sigma$ is a realization of $H$ if and only if \textcolor{black}{the SLS $R_{\Sigma}$ is a realization of $\Psi_H$} \cite{Petreczky2016}. The correspondence between $\Sigma$ and $R_{\Sigma}$ is a one-to-one \textcolor{black}{map between LPV-SSA/LSS realizations of $H$ and SLS realizations of $\Psi_H$, and it} preserves span-reachability, observability, minimality, and isomorphism \cite{Petreczky2016}.
  \textcolor{black}{Hence,} $H$ has a realization by an LPV-SSA/LSS if and only if $H$ is IRR and $\Psi_H$ has a  finite Hankel-rank, $\Sigma$ is a minimal realization of $H$ if and only if it is span-reachable, observable, and any two minimal LPV-SSA/LSS realizations of $H$ are isomorphic \cite{Petreczky2016}. 

  The LPV-SSA/LSS  \eqref{eq1:lpv} is called \emph{globally uniformly exponentially stable (GUES)} if there exist constants \textcolor{red}{$C \ge 1$}, $\beta \in (0,1)$ such that 
  \textcolor{black}
  {\(
  \|A(\underline{p}[k]) \cdots A(\underline{p}[1])\|_2 < C \beta^k
  \),
  for any  $\underline{p} \in \mathbb{P}^k$, $k >0$, 
  This is equivalent to \eqref{eq1:lpv} 
  being GUES at the zero equilibrium point for $v=0$, 
  and can be characterized by}
  the joint spectral radius of $\{A(p)\}_{p \in \mathbb{P}}$ being smaller than $1$. \textcolor{black}{As} $\{A_i\}_{i=1}^{n_p} \subseteq \{A(p)\}_{p \in \mathbb{P}} \subseteq Conv(\{A_i\}_{i=1}^{n_p})$, \textcolor{black}{the latter} is equivalent to SLS $R_{\Sigma}$ being GUES \cite{RJungers}.
  %

  The IRR of $H$ is called
  \emph{uniformly decaying (IRR-UD)}, if 
   \begin{align*}
   & \lim_{N \rightarrow \infty} \sup_{\underline{p} \in \mathbb{P}^{\infty}, \tau \in \mathbb{N},\tau \ge 1}\|(h \diamond \underline{p})(N+\tau,\tau)\|_2=0, 
   \end{align*}
  and \emph{uniformly exponentially decaying (IRR-UED) with rate $\beta \in (0,1)$} if 
  for some $C > 0$ \textcolor{black}{, for all $N \ge 1$,} 
  \begin{align*}
   \sup_{\underline{p} \in \mathbb{P}^{\infty}, \tau \in \mathbb{N}, \tau \ge 1} \|(h \diamond \underline{p})(N+\tau,\tau)\|_2 < C\beta^{\textcolor{black}{N-1}}.
   \end{align*}
  
  Intuitively, IRR-UD (IRR-UED) means that the impulse response $h \diamond \underline{p}$ decays (exponentially) 
  to zero, and it is equivalent to the sub-Markov parameters decaying (exponentially).
  \begin{Lemma}
  \label{lem:lpv:IRR-UD}
  $H$ is IRR-UD (resp. IRR-UED \textcolor{black}{with rate $\beta$}) if and only if $\Psi_H$ is UD 
  (resp. UED \textcolor{black}{with rate $\beta$}). 
  \end{Lemma}
  \begin{proof}[Proof of Lemma \ref{lem:lpv:IRR-UD}]
   We show that 
      \begin{equation}
      \label{eq:pf:lpv1}
      \begin{split}
          \sup_{\underline{p} \in \mathbb{P}^{\infty}}\!\! 
          \|(h \diamond \underline{p})(N+\tau,\tau)\|_2
          \!\! = \!\! 
          \sup_{s \in Q^{N+1}}  \|S(s)\|_2 
      \end{split}
      \end{equation}
   for any \textcolor{black}{$N \ge 1,\tau \in \mathbb{N}$}.  
      The statement of the lemma follows from \eqref{eq:pf:lpv1}
      \textcolor{black}{by noticing that 
       $\frac{1}{\sqrt{m}} \|S(qwq_1)\|_2  \le \sup_{l=1,\ldots,m} \|M_{q,l}(w)\|_2 \le \sqrt{n_p} \sup_{\sigma  \in Q} \|S(q w\sigma)\|_2$, $q,q_1 \in Q,w \in Q^{*}$.}
        To show \eqref{eq:pf:lpv1}, for any choice of $\underline{q} \in Q^{\infty}$, and any $N$, choose $\underline{p} \in \mathbb{P}^{\infty}$ such that $\underline{p}[i]=e_{\underline{q}[i]}$ \textcolor{black}{for all $i \ge 1$}. Then 
        $(h \diamond \underline{p})(N+\tau,\tau)=S(\textcolor{black}{\underline{q}[\tau] \cdots \underline{q}[N+\tau]})$
        and hence the left-hand side of \eqref{eq:pf:lpv1} is greater than or equal to the right-hand side.
        Conversely, for any choice of $\underline{p} \in \mathbb{P}^{\infty}$
	  by using the fact that $\mathbb{P} \subseteq \Delta$\textcolor{red}{,}  for any $N$, 
        \[ 
          \begin{split}
             & \|(h \diamond \underline{p})(N+\tau,\tau)\|_2  \le \sum_{s \in Q^{N+1}} \|S(s)\|_2 \underline{p}_s(N+\tau) \\
            & \le \sup_{s \in Q^{N+1}} \|S(s)\|_2 \sum_{s \in Q^{N+1}} \underline{p}_s(N+\tau) =
            \sup_{s \in Q^{N+1}} \|S(s)\|_2 \textcolor{black}{,}
          \end{split}
        \]
        where we used that  $\sum_{s \in Q^l} \underline{p}_s(t) = 1$ for any $t \ge l$.
  \end{proof}
  Theorem \ref{main:theo} applied to $\Psi_H$ and SLSs associated with LPV-SSAs results in the following.
      \begin{Corollary}
        \label{cor:lpv}
        Let $H$ be a response map \textcolor{black}{with} an IRR. 
        \\
        \textbf{(A)}
         If $\Psi_H$ has a finite Hankel-rank, then
         \textcolor{black}{IRR-UD and IRR-UED are equivalent.}
         \\
        \textbf{(B)}
         $H$ has a realization by a GUES LPV-SSA/LSS if and only  if 
         $\Psi_H$ has a finite Hankel-rank, and 
         \textcolor{black}{IRR-UD holds.}
         \\
        \textbf{(C)}
         $H$ is IRR-UED with the rate $\beta$ if and only if all minimal 
         LPV-SSA/LSS realizations of $H$ are GUES with the rate $\beta$.
      \end{Corollary}
      That is, the existence of a GUES LPV-SSA/LSS realization is equivalent to the existence of an IRR which is decaying, necessarily exponentially, 
      and the decay rate for IRR is the same as that of any minimal realization.

      \textcolor{black}{Next, 
      we show that BIBO stability together with $\ell_{\infty}$-forgetting is equivalent to the existence of a GUES realization.}
  We say that $H$ has the \emph{$\ell_{\infty}$-forgetting property} if there exists a decreasing sequence $\{G_{H,k}\}_{k=0}^{\infty}$ 
  \textcolor{black}{of real numbers}
  \textcolor{black}{such that 
  \textbf{(1)} $\lim_{k \rightarrow \infty} G_{H,k}=0$ and
  \textbf{(2)} for any
  scheduling signal $\underline{p} \in (\mathbb{P})^{t}$, $t \ge 1$ and 
  any
  two input signals $\underline{v}_1,\underline{v}_2  \in (\mathbb{R}^{m})^{t}$
  $i=1,2$ 
  such that $\underline{v}_1$ and $\underline{v}_2$} are the same for the last \textcolor{black}{$0 \le k \le t$} time steps, i.e., $\underline{v}_1(s)=\underline{v}_2(s)$ for $s=t-k+1,\ldots t$, \textcolor{black}{ the following holds with 
  $w_i=((\underline{p}[1],\underline{v}_i[1]) \cdots (\underline{p}[t],\underline{v}_i[t]))$, 
  $i=1,2$:}
  \begin{align}
  & \|H(w_1)\|_2 \le G_{H,0} \|\textcolor{black}{\underline{v}_1}\|_{\infty} \label{eq:l_infty_gain1} \textcolor{black}{,} \\
  & \|H(w_1) - H(w_2)\|_2 \le G_{H,k} \|\textcolor{black}{\underline{v}_1}-\textcolor{black}{\underline{v}_2}\|_{\infty}
  \label{eq:l_infty_gain2} \textcolor{black}{.}
  \end{align}
  We say that $H$ is \emph{exponentially  $\ell_{\infty}$-forgetting} \textcolor{black}{with rate $\beta \in (0,1)$}, if, \textcolor{black}{in addition to \eqref{eq:l_infty_gain1}-\eqref{eq:l_infty_gain2},} $G_{H,k}<\textcolor{black}{K}\beta^k$ \textcolor{black}{holds} for some $\textcolor{black}{K} > 0$ and all $k \in \mathbb{N}$. 
    \textcolor{black}{
    Intuitively, \eqref{eq:l_infty_gain1} ensures BIBO stability by giving a uniform $\ell_\infty$-gain bound, while \eqref{eq:l_infty_gain2} ensures that the effect of inputs more than $k$ steps in the past decays uniformly. Note that \eqref{eq:l_infty_gain1} alone implies absolute summability and forgetting for each fixed scheduling, but not uniformly across all scheduling signals \cite[Chapter~7d, Thm.~17]{Cal:Des}.}

  \begin{Theorem}
  If $H$ is IRR and has a finite Hankel-rank, then the following are equivalent: \textbf{(1)} $H$ has the $\ell_{\infty}$-forgetting property, \textbf{(2)} $H$ is exponentially $\ell_{\infty}$-forgetting, \textbf{(3)} $H$ is IRR-UED, \textbf{(4)} $H$ has a realization by a GUES LPV-SSA/LSS. 
  \textcolor{black}{Moreover, }if $H$ \textcolor{black}{is} $\ell_{\infty}$-forgetting, then 
  it \textcolor{black}{is} exponentially $\ell_{\infty}$-forgetting with a rate $\beta$,
  and 
  all minimal LPV-SSA/LSS realizations of $H$ are GUES with \textcolor{black}{the rate $\beta$}. 
  \end{Theorem}
  \begin{proof}
    The \textcolor{black}{theorem follows} from Corollary \ref{cor:lpv} and the following implications:

    \textbf{$\ell_{\infty}$-forgetting property $\implies$ $H$ is IRR-UD.}
      For \textcolor{black}{a} $\underline{p} \in \mathbb{P}^{\infty}$, \textcolor{black}{$k \ge 1$},
      define the response map
        \( H_{\underline{p},k}:(\mathbb{R}^m)^{+} \to \mathbb{R}^{\textcolor{black}{n_y}}\)
        as follows: for any $\underline{v} \in (\mathbb{R}^m)^{t}$,
        $t \ge 1$,
          let $w \in (\mathbb{P} \times \mathbb{R}^m)^{t+k}$ be such that
          $w[\tau]=(\underline{p}[\tau],\underline{v}[\tau])$ for $\tau=1,\ldots, t$, and $w[\tau]=(\underline{p}[\tau],0)$ for $\tau=t+1,\ldots, t+k$. Then we set
          $H_{\underline{p},k}(\underline{v}[1] \cdots \underline{v}[t]) = H(w)$. 
          It follows that 
          \textcolor{black}{for any sequences $\underline{v}_1,\underline{v}_2 \in (\mathbb{R}^m)^{t+k}$ such that 
          $\underline{v}_1[s]-\underline{v}_2[s]=\left\{\begin{array}{rl} \underline{v}[s] & s \le t \\ 0 & s > t \end{array}\right.$}
          \begin{equation}
            \label{eq:F_p_k}
             H_{\underline{p},k}(\underline{v})
            \!\!=\!\! \sum_{\tau=1}^{t} (h \diamond \underline{p})(t+k,\tau)\underline{v}[\tau]=H(w_1) \!\!-\!\! H(w_2) \textcolor{black}{,}
            \end{equation}
          \textcolor{black}{where $w_i=(\underline{p}[1],\underline{v}_i[1]) \cdots (\underline{p}[t+k],\underline{v}_i[t+k])$,}
          $i=1,2$.

          \begin{color}{black}
          For any output index $r=1,\ldots,n_y$ let $\underline{v}^{\textcolor{black}{r}} \in (\mathbb{R}^m)^t$ be such that 
          $\textcolor{black}{(\underline{v}^r[\tau])_j}=\left\{\begin{array}{rl} 1 & ((h \diamond \underline{p})(t+k,\tau))_{r,j}  > 0\\ -1 & \mbox{otherwise} \end{array}\right. $ is the sign of the $(r,j)$th entry of $(h \diamond \underline{p})(t+k,\tau)$ for
          $j=1,\ldots,m$.
          Then, 
         \( (H_{\underline{p},k}(\underline{v}^{\textcolor{black}{r}}))_r=\sum_{\tau=1}^{t} \sum_{j=1}^{m} |\left((h \diamond \underline{p})(t+k,\tau)\right)_{r,j}| \).  
           Hence 
          as $H$ is  $\ell_{\infty}$-forgetting, 
          \textcolor{black}{by setting $\underline{v}_2=0$ and $\underline{v}_1[s]=\underline{v}^{\textcolor{black}{r}}[s]$ for $s=1,\ldots,t$, it follows that} 
          $\textcolor{black}{(H_{\underline{p},k}(\underline{v}^{\textcolor{black}{r}}))_r \le} \|H_{\underline{p},k}(\underline{v}^{\textcolor{black}{r}})\|_{2}=\|H(w_1)-H(w_2)\|_2 \le G_{H,k}\|\underline{v}^{\textcolor{black}{r}}\|_{\infty} \textcolor{black}{\le G_{H,k}}$, and thus
          \begin{align*}
           & \sum_{\tau=1}^{t} \|(h \diamond \underline{p})(t+k,\tau)\|_2 \le  
           \sum_{\tau=1}^{t} \|(h \diamond \underline{p})(t+k,\tau)\|_F 
           \\
           & \le \sum_{\tau=1}^{t} \sum_{r=1}^{\textcolor{black}{n_y}} \sum_{j=1}^{m} |\left((h \diamond \underline{p})(t+k,\tau)\right)_{r,j}| \le 
         n_y G_{H,k} \textcolor{black}{.}
          \end{align*}
         As $\lim_{k \rightarrow \infty} G_{H,k}=0$,
         it then follows that $(h \diamond \underline{p})$ is IRR-UD. 
        \end{color}
    \textbf{IRR-UED with rate $\beta$ $\implies$ $H$ exponential $\ell_{\infty}$-forgetting with rate $\beta$.}
         If $H$ is IRR-UED, then set 
         \[
		 G_{H,k}=\textcolor{red}{\sqrt{m}}\sup_{t \ge 1, \underline{p} \in \mathbb{P}^{\infty}} \sum_{\tau=1}^{\textcolor{black}{\min\{t,t+k-1\}}} \|(h \diamond \underline{p})(t+k,\tau)\|_2
         \]
         for all $k \ge 0$.
         \textcolor{black}{Since $H$ is IRR-UED}, it follows that 
         $\|(h \diamond \underline{p})(t+k,\tau)\|_2 \le C \beta^{t+k-\tau\textcolor{black}{-1}}$ for some $C > 0$ and $\beta \in (0,1)$, and hence $G_{H,k} \le K\beta^k$ for 
	 $K=\textcolor{red}{\sqrt{m}}C/\textcolor{black}{(\beta(1-\beta))}$.
         Recall the definition of $H_{\underline{p},k}$ from the previous implication.
         \textcolor{black}{By repeating the argument of \cite[Chapter 7d, Theorem 17]{Cal:Des},}
         it follows that 
         \textcolor{black}{$G_{H,0}$ satisfies \eqref{eq:l_infty_gain1}, and}
         the $\ell_{\infty}$-gain of $H_{\underline{p},k}$ is upper bounded by $G_{H,k}$, \textcolor{black}{$k \ge 1$,} \textcolor{black}{i.e., $\|H_{\underline{p},k}(\underline{v})\|_2 \le G_{H,k} \|\underline{v}\|_{\infty}$.}
         Then \eqref{eq:F_p_k} implies that $G_{H,k}$ satisfies \eqref{eq:l_infty_gain2},
         for $k \ge 1$. 
  \end{proof}

\section{Conclusions}
We characterized when an input--output family admits a stable finite-dimensional realization (SLS, SAS, LPV-SSA, LSS): for finite Hankel-rank, stability is equivalent to uniform decay/forgetting, and minimal realizations inherit the same decay rate.
A future direction is to develop the relationship with echo-state/fading-memory properties and input--output tests for existence of quadratic/polyhedral/SOS Lyapunov functions.

\textbf{Acknowledgements:} we thank the anonymous reviewers for numerous useful technical remarks and suggestions. 

 \bibliographystyle{IEEEtran}
\bibliography{petreczky_chapterV2}
\end{document}